\documentclass{article}
\usepackage{amsrefs}
\usepackage{amssymb}
\usepackage{amsmath}
\usepackage{amsthm}
\usepackage{graphicx}

\newtheorem{theorem}{Theorem}
\newtheorem{lemma}{Lemma} 
\newtheorem{definition}{Definition}
\newtheorem{corollary}{Corollary}
\DeclareMathOperator{\rk}{rk}
\DeclareMathOperator{\val}{val} \DeclareMathOperator{\Supp}{Supp} \DeclareMathOperator{\divisor}{div}
\newtheorem*{subject}{2000 Mathematics Subject Classification}
\newtheorem*{keywords}{Keywords}

\author{Marc Coppens\footnote{KU  Leuven, Technologiecampus Geel, Departement Elektrotechniek (ESAT),
Kleinhoefstraat 4, B-2440 Geel, Belgium; email: marc.coppens@kuleuven.be.
Partially supported by the FWO-grant 1.5.012.13N}}
\title{Free divisors on metric graphs}
\date{}

\begin{document}
\maketitle \noindent

\begin{abstract}
On a metric graph we introduce the notion of a free divisor as a replacement for the notion of a base point free complete linear system on a curve. By means of an example we show that the Clifford inequality is the only obstruction for the existence of very special free divisors on a graph. This is different for the situation of base point free linear systems on curves. It gives rise to the existence of many types of divisors on graphs that cannot be lifted to curves maintaining the rank and it also shows that classifications made for linear systems of some fixed small positive Clifford index do not hold (exactly the same) on graphs.
\end{abstract}

\begin{subject}
14H51; 14T05; 05C99
\end{subject}

\begin{keywords}
graph, free divisor, Clifford index
\end{keywords}

\section{Introduction}\label{section1}

The theory of divisors on a compact metric graph $\Gamma$ is  developed very similar to the theory of divisors on a smooth projective curve $C$.
In their paper \cite{ref3} the authors proved a Riemann-Roch theorem for metric graphs with a statement completely similar to the Riemann-Roch theorem for curves.
This was an extension of a similar result for finite graphs in \cite{ref2}.
In this paper, when talking about a graph we mean a metric graph.

Instead of the dimension of the complete linear system associated to a divisor on a curve one has the concept of the rank $\rk (D)$ of a divisor $D$ on a graph $\Gamma$.
On a graph $\Gamma$ there is a well-defined canonical divisor $K_{\Gamma}$ and an effective divisor $D$ on $\Gamma$ is called special in case the rank of $K_{\Gamma}-D$ is at least $0$.
As in the case of curves the rank of a non-special divisor $D$ on $\Gamma$ is completely determined by its degree $\deg (D)$ because of the Riemann-Roch Theorem.
For a general effective special divisor $D$ either its rank or the rank of $K_{\Gamma}-D$ is zero.
An effective divisor $D$ such that both $\rk (D)>0$ and $\rk (K_{\Gamma}-D)>0$ is called very special and when considering the possible ranks of divisors one can restrict to very special divisors.
Such divisors have their degree between 2 and 2g-4.
If the genus of the graph is less than 2 then there are no very special divisors on the graph, so from now on we only consider graphs of genus $g$ at least 3.
It follows easily from the Theorem of Riemann-Roch that for a very special divisor $D$ on a metric graph $\Gamma$ one has $\deg (D)\geq 2\rk (D)$.
In the case of curves this corresponding statement is part of the so-called Clifford Theorem.

As in the case of curves, and for the same reason, this inequality is the only obstruction for the existence of a graph $\Gamma$ having a very special effective divisor $D$ of a prescribed degree $d$ between 2 and $2g-4$ and rank $r$.
In case $\Gamma$ is a hyperelliptic graph of genus $g$, meaning it has an effective divisor $E$ of degree 2 and rank 1, then $rE+F$ with $F$ a general effective divisor on $\Gamma$ of degree $d-2r$ is a very special divisor of degree $d$ and rank $r$ (of course we assume $2\leq d\leq 2g-4$ and $d\geq 2r$).
In case $D$ has rank $r$ and $P$ is any point on $\Gamma$ then also $D+P$ has rank at least $r$.
In the case of curves if a divisor $D$ satisfies $\dim (D)=r$ and $P$ is a point on $C$ such that $\dim (D+P)=r$ the $P$ is a fixed point of the complete linear system $\vert D+P\vert$, meaning every divisor of $\vert D+P\vert$ contains $P$.
In the relationship between linear systems on curves and projective realisations of the curve one is only interested in linear systems without base points.
The examples mentioned above in the case of hyperelliptic curves are not base point free in case $d-2r>0$. 
In the case one restricts to base point free linear systems on curves one gets more obstructions.
As an example, if $p$ is a prime number and $D$ is an effective divisor on a curve $C$ of degree $p$ such that $\vert D\vert$ is a base point free linear system of dimension 2 then $C$ is birationally equivalent to a plane curve of degree $p$ and therefore $g\leq p(p-1)/2$.
In the case of graphs, in case $D$ is an effective divisor of rank at least 1 such that for each point $P$ on $\Gamma$ one has $\rk (D-P)<\rk (D)$ then we say $D$ is a free divisor on $\Gamma$.
In the case of curves the corresponding definition is equivalent to $\vert D\vert$ being base point free.
In the case of graphs it could be that a divisor $D$ of rank at least 1 that is not free does not contain a point $P$ such that each effective divisor linearly equivalent to $D$ contains $P$.
In contrast to the case of curves we show that for each genus $g$ and for each degree $2\leq d\leq 2g-4$ and each rank $r$ with $d-2r\geq 0$ there exist a graph $\Gamma$ containing a very special effective free divisor $D$ of degree $d$ and rank $r$.

We mention a few consequences concerning the differences between the theories of divisors on curves and on graphs.
In his paper \cite{ref8} the author describes a degeneration such that given a divisor $D$ of degree $d$ with $\dim (D)=r$ on a curve $C$ one obtains a graph $\Gamma$ having a divisor $\tau (D)$ of degree $d$ and rank at least $r$.
Given an effective divisor $D$ on a graph $\Gamma$ of degree $d$ and rank $r$ one is interested in conditions implying it comes from a divisor on a curve with the same rank and dimension.
In \cite{ref10} and \cite{ref11} it is shown that this correspondence between divisors on graphs and curves is very good in the case of graphs of genus at most 3 and in case of hyperelliptic graphs.
Moreover for every genus $g$ in \cite{ref12} one obtains the existence of graphs for which this corresponde is very good.
However in \cite{ref13} one obtains graphs of genus 4 such that not all effective divisors can be lifted to the same curve such that the rank on the graph is equal to the dimension of the associated complete linear system on the curve.
More examples of this kind are obtained, see e.g. \cite{ref14}.
The examples in this paper show the existence of types of divisors on graphs that cannot be lifted because such divisors do not exist on curves.
As an example, for all genus $g\geq 7$ we obtain graphs $\Gamma$ having a free divisor of degree 5 and rank 2.
If such a divisor would be a specialisation of a divisor $D$ on a curve $C$ then $\mid D\mid$ has to be a base point free linear system on $C$ of degree 5  and dimension at least 2. Such curves do not exist.

Associated to a very special divisor $D$ is the Clifford index of $D$ defined by $c(D)=\deg (D)-2\rk (D)$.
The inequality of Clifford is equivalent to $c(D)\geq 0$.
In the case of curves the second part of Clifford's Theorem is the classification of very special divisors of Clifford index 0.
This classification also holds for graphs.
In \cite{ref15} it is proved that in case a graph has a very special divisor of Clifford index 0 then the graph is hyperelliptic.
Moreover, it is proved in \cite{ref15} that in case $\Gamma$ is a hyperelliptic graph and $E$ is a divisor of degree 2 and rank 1 on $\Gamma$ then each very special divisor on $\Gamma$ of Clifford index 0 is linearly equivalent to a multiple of $E$.
In the case of curves, for small values of the Clifford index there exist classifications of divisors of the given Clifford index.
As a next case in \cite{ref16} it is proved that on a non-hyperelliptic curve each very special divisor of degree at most $g-1$  of Clifford index 1 has dimension at most 2 (and in case it has dimension 2 then $g\leq 6$).
In this paper for all genus $g\geq 3$ and all $1\leq r\leq g-2$ we show the existence of a graph $\Gamma$ of genus $g$ having a divisor $D$ of degree $2r+1$ and rank $r$ (hence $c(D)=1$).
Therefore the classifications of linear systems on curves of small Clifford index do not hold for graphs in case of non-zero Clifford index.

\section{Generalities}\label{section2}
A \emph{metric graph} $\Gamma$ (shortly a graph) is a compact connected metric space such that for each point $P$ on $\Gamma$ there exists $n\in \mathbb{Z}$ with $n\geq 1$ such that some neighbourhood of $P$ is isometric to $\{ z\in \mathbb{C} : z=te^{2k\pi /n} \text{ with } t\in [0,r]\subset \mathbb{R} \text { and } k \in \mathbb{Z} \}$ for some $r> 0$ and with $P$ corresponding to 0.
The \emph{valence} $\val (P)$ is equal to $n$.
We say that $P$ is a \emph{vertex} of $\Gamma$ if $\val (P) \neq 2$.

Let $V(\Gamma )$ be the set of vertices of $\Gamma$.
A connected component $e$ of $\Gamma \setminus V(\Gamma )$ is called an \emph{edge} of $\Gamma$.
We write $\overline{e}$ to denote its closure and we call it a closed edge of $\Gamma$.
In case $\overline{e} \setminus e$ is just one vertex of $\Gamma$ then we call $\overline {e}$ a \emph{loop}.
In case $V(\Gamma )=\emptyset$ then $\Gamma$ is also considered as being a closed loop.
Let $E(\Gamma )$ be the set of closed edges of $\Gamma$. The \emph{genus} of $\Gamma$ is defined by

\[
g(\Gamma ) = \vert E(\Gamma ) \vert - \vert V(\Gamma ) \vert +1 \text { .}
\]

An \emph{arc} on $\Gamma$ based at $P$ is a metric map $\tau : [0,a] \rightarrow \Gamma$ with $\tau (0)=P$.
Two such arcs $\tau$ and $\tau '$ are called equivalent in case $\tau (\varepsilon) = \tau ' (\varepsilon) $ for $\varepsilon >0$ very small.
An equivalence class of such arcs is called a \emph{tangent direction} of $\Gamma$ at $P$ and $T_P(\Gamma )$ is the set of tangent directions of $\Gamma$ at $P$.
Clearly $T_P (\Gamma )$ has $\val (P)$ tangent directions.

A \emph{divisor} $D$ on a metric graph $\Gamma$ is a finite formal linear combination $\sum _{P \in \Gamma} n_PP$ of points on $\Gamma$ with integer coefficients (hence $n_P \neq 0$ for finitely many points $P$ an $\Gamma$).
We also write $D(P)$ to denote the coefficient $n_P$.
The \emph{degree} of $D$ is $\deg (D)=\sum _{P\in \Gamma} n_P$.
We say $D$ is \emph{effective} if $n_P\geq 0$ for all $P\in \Gamma$.
The \emph{canonical divisor} of $\Gamma$ is

\[
K_{\Gamma} = \sum _{P \in \Gamma}(\val (P)-2)P
\]
(since $\Gamma$ is compact one has $\val (P) \neq 2$ for finitely many points $P$ on $\Gamma$).
A \emph{rational function} on $\Gamma$ is a continuous function $f: \Gamma \rightarrow \mathbb{R}$ that can be described as a piecewise affine function with integer slopes on the edges.
For $P \in \Gamma$ we define $\divisor (f)(P)$ as being the sum of all slopes of $f$ on $\Gamma$ in all directions emanating from $P$.
In this way $f$ defines a divisor $\divisor (f)=\sum _{P \in \Gamma} \left( \divisor (f)(P) \right) P$.
Two divisors $D_1$ and $D_2$ are called \emph{linearly equivalent} if $D_2 - D_1= \divisor (f)$ for some rational function $f$ on $\Gamma$.
For a divisor $D$ we define the \emph{rank} $\rk (D)$ as follows.
In case $D$ is not linearly equivalent to an effective divisor then $\rk (D)=-1$.
Otherwise $\rk (D)$ is the minimal number $r$ such that for each effective divisor $E$ of degree $r$ on $\Gamma$ there exists an effective divisor $D'$ on $\Gamma$ linearly equivalent to $D$ and containing $E$.

\subsection{Reduced divisors}\label{subsection2.1}

The concept of reduced divisors is introduced in \cite{ref2} in the context of finite graphs.
Its generalization to the metric case is introduced at different papers (see e.g. \cite{ref4}, \cite{ref7}).
Let $D$ be a divisor on a metric graph $\Gamma$, let $X$ be a closed subset of $\Gamma$ and let $P \in \partial X$ be a boundary point of $X$.
We say a tangent direction $v \in T_P(\Gamma )$ leaves $X$ at $P$ if for an arc $\tau$ representing $v$ and $\varepsilon >0$ small one has $\tau (\varepsilon) \notin X$.
We say $P$ is a \emph{saturated boundary point} of $X$ with respect to $D$ if the number of tangent directions $v$ at $P$ leaving $X$ is at most $D(P)$, otherwise we call it \emph{non-saturated}.

\begin{definition}\label{definition1}
Let $D$ be a divisor on a metric graph $\Gamma$ and let $P_0\in \Gamma$.
We say $D$ is a $P_0$-reduced divisor in case the following conditions are satisfied
\begin{enumerate}
\item $\forall P \in \Gamma \setminus \{ P_0\} : D(P)\geq 0$
\item $\forall X \subset \Gamma$, a closed subset with $P_0 \notin X$, there exists $P \in \partial X$ such that $P$ is non-saturated with respect to $D$.
\end{enumerate}
\end{definition}

\begin{theorem}\label{theorem1}
Let $D$ be a divisor on a metric graph $\Gamma$ and let $P \in \Gamma$.
There is a unique $P$-reduced divisor $D_P$ on $\Gamma$ linearly equivalent to $D$.
\end{theorem}

\subsection{The burning algorithm}\label{subsection2.2}

In \cite{ref6}, an algorithm is described such that, given an effective divisor $D$ on $\Gamma$ and $P \in \Gamma$, one finds the $P$-reduced divisor $D_P$ linearly equivalent to $D$.
From this algorithm it is clear that $D_P(P)$ is the maximal value $m=D'(P)$ for an effective divisor $D'$ linearly equivalent to $D$.

Part of that algorithm consists of checking that a given effective divisor $D$ is $P$-reduced. This can be made very visible using the so-called \emph{burning algorithm} (see also \cite{refextra}*{A.3}).

In case $Q \in \Supp (D)$ with $Q \neq P$ then we assume there are $D(Q)$ fire-fighters available at $Q$.
There is a fire starting at $P$ and following all tangent directions of $P$.
Each time the fire reaches some point $Q \in \Supp (D)$ along some tangent direction $v$ at $Q$ then one fire-fighter does stop the fire at $Q$ coming in along $v$.
Then this fire-fighter is occupied and not available any more.
In case the fire reaches some point $Q \in \Gamma$ and there is no more fire-fighter available at $Q$ (because $Q \notin \Supp (D)$ or all fire-fighters at $Q$ are occupied by fires coming in from other tangent directions at $Q$) then $Q$ gets burned and the fire goes on from $Q$ in all tangent directions not yet burned.
The divisor is $P$-reduced if and only if the whole graph gets burned by this fire.

\subsection{The Jacobian and Abel-Jacobi maps}\label{subsection2.3}

For making some dimension arguments we use the tropical Jacobian of a metric graph and the associated Abel-Jacobi maps defined in \cite{ref4}. In that paper the authors do introduce one-forms on $\Gamma$.
The space $\Omega (\Gamma)$ of those one-forms is a $g$-dimensional real vectorspace.
Similar to the case of curves integrals along paths on $\Gamma$ do define linear functions on $\Omega (\Gamma)$. The cycles on $\Gamma$ do define a lattice $\Lambda$ in the dual space $\Omega (\Gamma )^*$ and we call $J(\Gamma )=\Omega (\Gamma )^*/ \Lambda$ the \emph{tropical Jacobian} of $\Gamma$.

Fixing a base point $P_0$ on $\Gamma$ one also obtains for each non-negative integer $d$ an \emph{Abel-Jacobi map} $I(d) : \Gamma ^{(d)} \rightarrow J(\Gamma )$.
Here $\Gamma ^{(d)}$ is the $d$-the \emph{symmetric product} of $\Gamma$ parametrizing effective divisors of degree $d$ on $\Gamma$.
The structure of $\Gamma ^{(d)}$ is described in e.g. \cite{ref9}.
For $D_1, D_2 \in \Gamma^{(d)}$ one has $I(d)(D_1)=I(d)(D_2)$ if and only if $D_1$ and $D_2$ are linearly equivalent.
For $d\geq g$ the map $I(d)$ is surjective.
For $d<g$ the image of $I(d)$ is a finite union of images on $J(\Gamma )$ of bounded open subsets of affine subspaces of $\Omega (\Gamma)^*$ of dimension at most $d$.

\section{The example}\label{section3}

First we recall the definition of a free divisor on a graph $\Gamma$.

\begin{definition}\label{definition2}
An effective divisor $D$ on a metric graph $\Gamma$ is called free if for all $P\in \Gamma$ one has $\rk (D-P)< \rk (D)$.
\end{definition}

In this section we are going to prove the following theorem
\begin{theorem}\label{theorem2}
For all integers $r$, $d$ and $g$ satisfying $r\geq 1$, $d\leq r+g-2$ and $d-2r\geq 0$ there exists a metric graph $\Gamma$ of genus $g$ such that $\Gamma$ has a free divisor $D$ of degree $d$ with $\rk (D)=r$.
\end{theorem}

In case $d-2r=0$ this is well-known (see the Introduction), so we can assume $d-2r \geq 1$.
Writing $d=2r+a$ we obtain $a\geq 1$ instead of the condition $d-2r \geq 1$ and the condition $g \geq r+a+2$ instead of $d\leq r+g-2$. In particular we can assume $g \geq a+3 \geq 4$.
For a fixed value of the integer $a$ we are going to construct a graph $\Gamma$ such that for all integers $r$ satisfying $1\leq r\leq g-a-2$ there exists a free divisor $D$ on $\Gamma$ of degree $2r+a$ such that $\rk (D)=r$.

To construct the graph we start with a graph $\Gamma_a$ consisting of $a$ circles as used in \cite{ref17} (see figure \ref{Figuur 1}).
\begin{figure}[h]
\begin{center}
\includegraphics[height=2 cm]{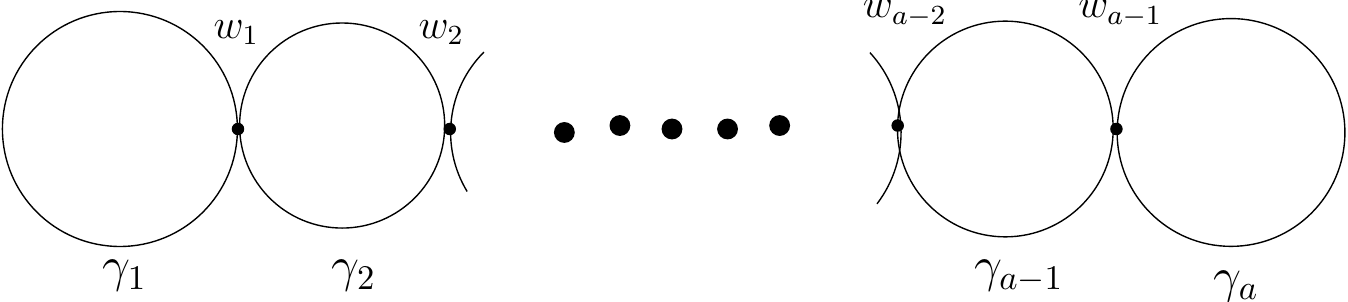}
\caption{the graph $\Gamma_a$}\label{Figuur 1}
\end{center}
\end{figure}
We do not need all assumptions from \cite{ref17}*{Definition 4.1}, we only need that the lengths of both edges from $w_i$ to $w_{i+1}$ are different for $1\leq i\leq a-2$.
Take $v_g \in \gamma_1\setminus \{w_1 \}$, $v_{g-1} \in \gamma_a \setminus \{ w_{a-1} \}$ such that $2v_g$ is not linearly equivalent to $2w_1$ on $\gamma_1$ and $2v_{g-1}$ is not linearly equivalent to $2w_{a-1}$ on $\gamma_a$ and some more different points $v_{a+1}, \cdots , v_{g-2}$ (remember $g\geq a+3$) on $\Gamma_0$ different from $w_1, \cdots, w_{a-1}$.
For each $a+1\leq i\leq g$ we attach a loop $\gamma_i$ to $\Gamma_0$ at $v_i$. This is the graph $\Gamma$ of genus $g$ we are going to use (see figure \ref{Figuur 2}).
\begin{figure}[h]
\begin{center}
\includegraphics[height=8 cm]{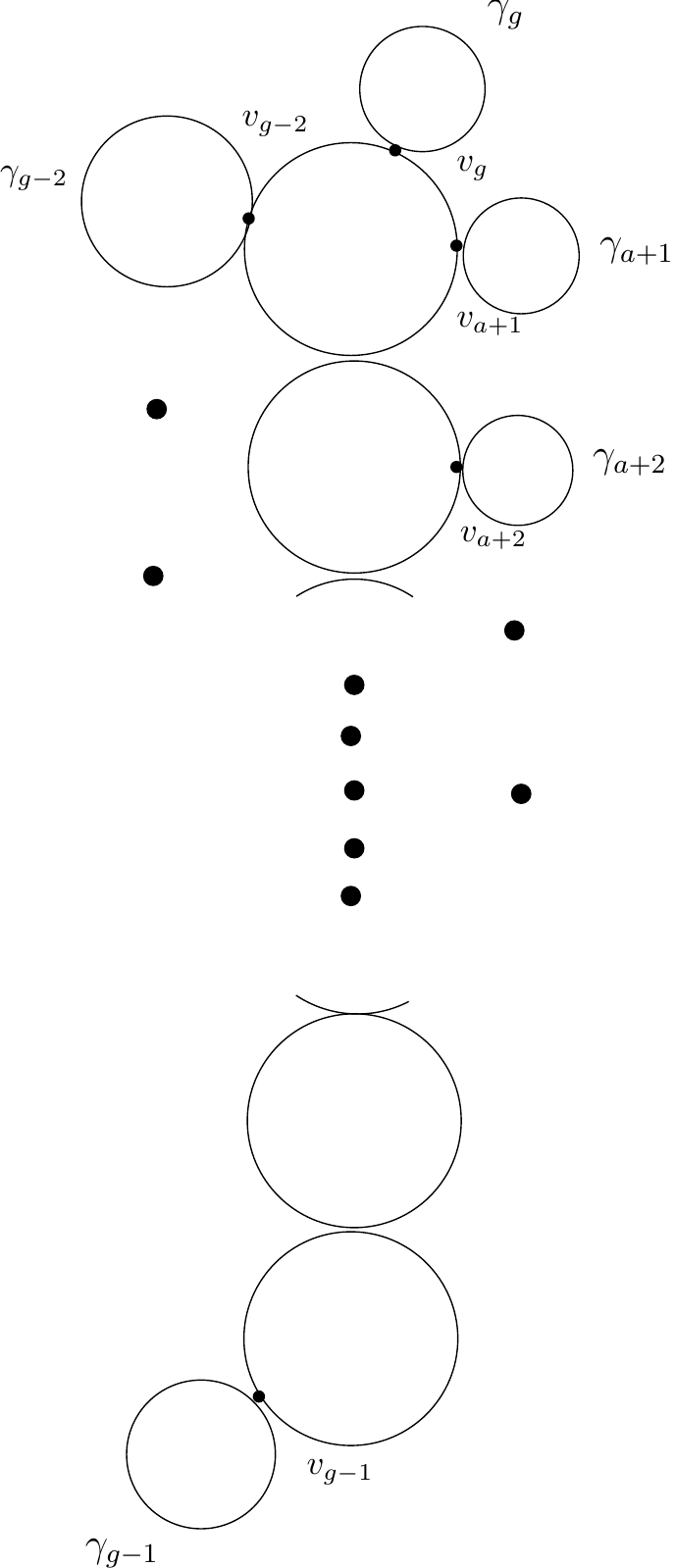}
\caption{the graph $\Gamma$}\label{Figuur 2}
\end{center}
\end{figure}
A divisor $E$ on $\Gamma_0$ or on $\gamma_i$ for some $a+1\leq i \leq g$ can also be considered as a divisor on $\Gamma$.
Therefore we write $\rk _{\Gamma_0} (E)$ (resp. $\rk _{\gamma_i}(E)$, $\rk _{\Gamma}(E)$) to denote its rank as a divisor on $\Gamma_0$ (resp. $\gamma_i$, $\Gamma$).

\begin{lemma}\label{lemma2}
In case $E$ and $E'$ are linearly equivalent divisors on $\Gamma _0$ or $\gamma _i$ for some $a+1 \leq i \leq g$ then $E$ and $E'$ are also linearly equivalent divisors on $\Gamma$.
\end{lemma} 
\begin{proof}
Assume $E$ and $E'$ are linearly equivalent divisors on $\Gamma _0$.
Let $f$ be a rational function on $\Gamma _0$ such that $\divisor (f)=E-E'$.
Let $g$ be the rational function on $\Gamma$ such that its restriction to $\Gamma_0$ is equal to $f$ and for $a+1 \leq i \leq g$ its restriction is the constant function with value $f(v_i)$, then on $\Gamma$ one has $\divisor (g)=E-E'$. This proves the lemma in the case of $\Gamma _0$, the other case is similar.
\end{proof}

\begin{lemma}\label{lemmaextra1}
Let $E$ and $E'$ be effective divisors on $\Gamma_0$ such that $E$ and $E'$ are linearly equivalent as divisors on $\Gamma$. Then $E$ and $E'$ are linearly equivalent as divisors on $\Gamma$.
\end{lemma}
\begin{proof}
There is a rational function $f$ on $\Gamma$ such that $\divisor (f)=E-E'$.
For $a+1 \leq i \leq g$ there is no point $P$ on $\gamma_i \setminus \{v_i \}$ contained in the support of $\divisor (f)$.
Since $\gamma_i$ is a loop this implies $f$ is constant on $\gamma_i$.
Let $f'$ be the restriction of $f$ to $\Gamma_0$ then this implies $\divisor (f')=E-E'$ on $\Gamma_0$, hence $E$ and $E'$ are linearly equivalent divisors on $\Gamma_0$.
\end{proof}

\subsection{Effective divisors of degree $a+2r$ and rank $r$}\label{subsection3.1}

\begin{lemma}\label{lemma1}
Let $D$ be an effective divisor of degree $a+2r$ on $\Gamma_0$ with $1\leq r\leq g-a-2$ then $\rk_{\Gamma}(D)\geq r$.
\end{lemma}
\begin{proof}
Let $E=P_1+ \cdots + P_r$ be an effective divisor on $\Gamma$.
We need to prove that there exists an effective divisor $D'$ on $\Gamma$ linearly equivalent to $D$ containing $E$.
From the Riemann-Roch Theorem it follows $\rk_{\Gamma_0}(D)\geq 2r$, hence in case $P_i \in \Gamma_0$ for $1 \leq i \leq r$ then there exists a divisor $D'$ on $\Gamma _0$ linearly equivalent to $D$ containing $E$.
From Lemma \ref{lemma2} we know $D'$ is also linearly equivalent to $D$ on $\Gamma$.

So we can assume not all points $P_i$ belong to $\Gamma_0$.
Define $a+1 \leq i_1 < i_2 < \cdots < i_{r'} \leq g$ such that $E \cap (\gamma_j \setminus \{ v_j \}) \neq  \emptyset$ for some $a+1 \leq j \leq g$ if and only if there exists $1 \leq k\leq r'$ such that $i_k=j$.
Let $\deg(E \cap (\gamma_{i_k} \setminus \{ v_{i_k} \}))=n_k$ for $1 \leq k \leq r'$ and let $\deg ( E \cap \Gamma_0)=n_0$.
On $\Gamma_0$ we consider the divisor $E_0=(E \cap \Gamma_0)+ (n_1+1)v_{i_1} + \cdots + (n_{r'}+1)v_{i_{r'}}$.
It has degree $n_0+(n_1+1)+\cdots (n_{r'}+1)=r+r'\leq 2r$ hence there exists a divisor $D''$ on $\Gamma_0$ linearly equivalent to $D$ containing $E_0$.
From Lemma \ref{lemma2} it follows $D''$ is linearly equivalent to $D$ as a divisor on $\Gamma$.
For $1 \leq k \leq r'$ there exists a point $Q_k$ on $\gamma_{i_k}$ such that $(E \cap (\gamma_{i_k} \setminus \{ v_{i_k} \}) +Q_k$ is linearly equivalent to $(n_k+1)v_{i_k} $ on $\gamma_{i_k}$.
Again because of Lemma \ref{lemma2} linearly equivalence between those divisors also holds on $\Gamma$.
Therefore on $\Gamma$ the divisor $D$ is linearly equivalent to $D''-\sum_{k=1}^{r'}(n_k+1)v_{i_k}+\sum_{k=1}^{r'}(E \cap (\gamma_{i_k} \setminus \{ v_{i_k} \}) +Q_k$.
This is an effective divisor $D'$ on $\Gamma$ containing $E$.
\end{proof}

In order to finish the proof of Theorem \ref{theorem2} it is enough to prove that there exists an effective divisor $D$ on $\Gamma_0$ of degree $a+2r$ such that for each point $P$ on $\Gamma$ one has $\rk_{\Gamma} (D-P)<r$.
Indeed because of Lemma \ref{lemma1} it would imply $\rk_{\Gamma} (D)=r$ and $D$ is a free divisor on $\Gamma$.
We are going to show that the subset of $J(\Gamma_0)$ consisting of the images under $I(a+2r)$ of effective divisors $D$ on $\Gamma_0$ of degree $a+2r$ such that there is a point $P$ on $\Gamma$ with $\rk_{\Gamma} (D-P) \geq r$ is different from $J(\Gamma_0)$.
Since the map $I(a+2r) : \Gamma_0^{(a+2r)} \rightarrow J(\Gamma_0)$ is surjective this is enough to finish the proof of Theorem \ref{theorem2}.

\subsection{The case $P \notin \Gamma_0$}\label{subsection3.2}

\begin{lemma}\label{lemma3}
Let $D$ be an effective divisor of degree $a+2r$ on $\Gamma_0$ such that there exists a point $P \in \gamma_{a+1} \setminus \{ v_{a+1} \} $ satisfying $\rk _{\Gamma} (D-P)=r$ then there exists an effective divisor $E$ of degree $a-2$ on $\Gamma_0$ such that $D$ is linearly equivalent to $2(v_{a+1}+v_{a+2}+\cdots + v_{a+r+1})+E$ on $\Gamma_0$.
\end{lemma}

From Lemma \ref{lemma3} it follows that in case $D$ is an effective divisor of degree $a+2r$ on $\Gamma_0$ such that there exists a point $P \in \gamma_{a+1} \setminus \{ v_{a+1} \} $ satisfying $\rk_{\Gamma} (D-P)=r$ then $I(2r+a)(D)$ belongs to the subset $I(2r+2)(2(v_{a+1}+\cdots v_{a+r+1}))+I(a-2)(\Gamma_0^{(a-2)})$ of $J(\Gamma_0)$.
This subset of $J(\Gamma_0)$ is the union of finitely many images of bounded open subsets of affine subspaces of dimension at most $a-2$ of $\Omega (\Gamma)^*$.
A similar conclusion holds for $P$ on any $\gamma_i$ for $a+1\leq i\leq g$.
If $I(2r+a)(D)$ does not belong to such subset of $J(\Gamma)$ then $\rk _{\Gamma} (D-P)<r$ for all $P \notin \Gamma_0$.

In order to prove Lemma \ref{lemma3} we are going to use the following lemma and its corollaries.

\begin{lemma}\label{lemma4}
Let $G_0$ be a metric graph and let $G$ be the graph obtained from $G_0$ by attaching a loop $\gamma$ at some point $v \in G_0$ (see figure \ref{Figuur 3}). 
Let $P$ be a point of $\gamma \setminus \{ v \}$ and let $D$ be an effective divisor on $G_0$.
Let $Q$ be a point on $G_0$ and let $D_{Q,0}$ be the $Q$-reduced divisor on $G_0$ linearly equivalent to $D$ and let $D_Q$ be the $Q$-reduced divisor on $G$ linearly equivalent to $D+P$ then $D_Q=D_{Q,0}+P$.
\end{lemma}
\begin{figure}[h]
\begin{center}
\includegraphics[height=3 cm]{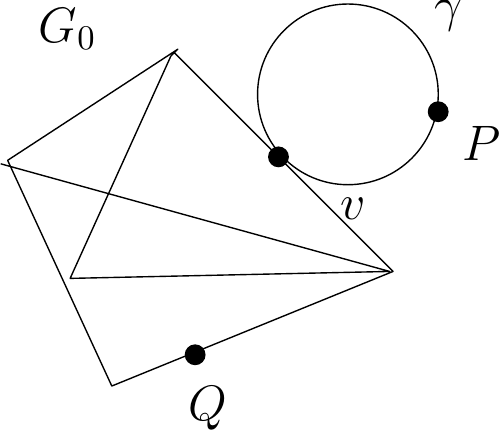}
\caption{the graph $\Gamma$}\label{Figuur 3}
\end{center}
\end{figure}
\begin{proof}
Since $D_{Q,0}$ is linearly equivalent to $D$ on $G_0$ it follows from Lemma \ref{lemma2} that $D_{Q,0}$ is linearly equivalent to $D$ on $G$, hence $D_{Q,0}+P$ is linearly equivalent to $D+P$ on $\Gamma$.
So it is enough to prove that $D_Q+P$ is a $Q$-reduced divisor on $G$.
This is clear using the burning algorithm: a fire starting at $Q$ burns the whole graph $G_0$ since $D_Q$ is a $Q$-reduced divisor on $\Gamma_0$.
Hence a fire leaves $v$ in both tangent directions to $\gamma$ at $v$.
These fires do reach $P$ along both tangent directions to $\gamma$ at $P$, hence also $P$ gets burned because there is only one fire-fighter at $P$.
This proves $D_Q+P$ is $Q$-reduced on $G$.
\end{proof}

\begin{corollary}\label{lemma5}
Let $G_0$ be a metric graph and let $G$ be the graph obtained from $G_0$ by attaching a loop $\gamma$ at some point $v \in G_0$. Let $P$ be a point of $\gamma \setminus \{ v \}$ and let $D$ be an effective divisor on $G_0$.
If $\rk_G(D+P)\geq r$ then $\rk_{G_0}(D)\geq r$.
\end{corollary}
\begin{proof}
Assume $\rk_G(D+P)\geq 1$ and let $Q$ be a  point on $G_0$.
Let $D_Q$ be the $Q$-reduced divisor linearly equivalent to $D$ on $G_0$.
From Lemma \ref{lemma4} we know $D_Q+P$ is the $Q$-reduced divisor on $G$ linearly equivalent to $D+P$.
Since $\rk_G(D+P)\geq 1$ it follows $D_Q+P$ contains $Q$, hence $D_Q$ contains $Q$.
This proves $\rk_{G_0}(D)\geq 1$.

Now assume $\rk_G(D+P)\geq r> 1$, by means of induction we can assume that $\rk_{G_0}(D) \geq r-1$.
Take $P_1+ \cdots + P_r$ on $G_0$.
We can assume $D=P_1+ \cdots P_{r-1}+D'$ for some effective divisor $D'$ on $G_0$.
Then $\rk_G (D'+P)\geq 1$ hence $\rk_{G_0}(D')\geq 1$.
Hence $D'$ is linearly equivalent to $P_r+D"$ on $G_0$ for some effective divisor $D"$ on $G_0$.
Then $D"+P_1+\cdots P_r$ is linearly equivalent to $D$ on $G_0$, hence $\rk_{G_0}(D)\geq r$.
\end{proof}

\begin{corollary}\label{corollary}
Let $G_0$ be a metric graph and let $G$ be the graph obtained from $G_0$ by attaching a loop $\gamma$ at some point $v \in G_0$. Let $D$ be an effective divisor of degree $d$ on $G_0$.
There do not exist an effective divisor $D'$ of degree $d-1$ on $G_0$ and $P \in \gamma \setminus \{ v \}$ such that $D$ and $D'+P$ are linearly equivalent divisors on $\Gamma$.
\end{corollary}

\begin{proof}
Assume $D'$ is an effective divisor of degree $d-1$ on $G_0$ and $P \in \gamma \setminus \{ v \}$ such that $D$ and $D'+P$ are linearly equivalent divisors on $G$.
Choose $Q \in G_0$ and let $D_Q$ (resp. $D'_Q$) be the $Q$-reduced divisors on $G_0$ linearly equivalent to $D$ (resp. $D'$) on $G_0$.
Because of Lemma \ref{lemma4} $D_Q$ (resp. $D'_Q+P$) is the $Q$-reduced divisor on $G$ linearly equivalent to $D$ (resp. $D'+P$) on $G$.
By assumption $D$ and $D'+P$ are linearly equivalent divisors on $G$.
Since $D_Q \neq D'_Q+P$ this contradicts Theorem \ref{theorem1}.
\end{proof}

Now we are going to prove Lemma \ref{lemma3}. On $\Gamma_0$ the divisor $D$ is linearly equivalent to $2v_{a+1}+D'$ for some effective divisor $D'$ of degree $a+2r-2$ on $\Gamma_0$.
On the loop $\gamma_{a+1}$ one has $2v_{a+1}$ is linearly equivalent to $P+P'$ for some point $P'$ different from $v_{a+1}$.
From Lemma \ref{lemma2} it follows $D$ is linearly equivalent to $D'+P+P'$ on $\Gamma$.
By assumption we have $\rk_{\Gamma}(D'+P')\geq r$.
Let $\Gamma '=\Gamma \setminus (\gamma_{a+1} \setminus \{ v_{a+1} \})$ then from Corollary \ref{lemma5} we obtain $\rk_{\Gamma '}(D')\geq r$.
For $a+2 \leq j \leq a+r+1$ let $P_j$ be a point on $\gamma_j$ different from $v_j$.
On $\Gamma'$ there exists an effective divisor $D''$ linearly equivalent to $D'$ containing $P_{a+2}+ \cdots + P_{a+r+1}$.
Write $D''=D''_{a+2}+\cdots +D''_g+E'$ with $\Supp(D''_j)\subset \gamma_j \setminus \{ v_j \}$ for $a+2 \leq j \leq g$ and $\Supp (E')\subset \Gamma_0$.
On $\gamma_j$ there exist a non-negative integer $r_j$ and an effective divisor $E''_j$ of degree at most 1 such that $D''_j$ is linearly equivalent to $r_jv_j+E''_j$.
From Lemma \ref{lemma2} we know that $D'''=2v_{a+1}+r_{a+2}v_{a+2}+\cdots+r_gv_g+E''_{a+2}+\cdots +E''_g+E'$ is linearly equivalent to $D$ as a divisor on $\Gamma$

From Corollary \ref{corollary} it follows $E''_{a+2}+\cdots +E''_g=0$.
However for $a+2\leq j\leq a+r+1$ one has $P_j \in D''_j$ and $D''_j$ is linearly equivalent to $r_jv_j$ on $\gamma_j$, hence $r_j\geq 2$.
Hence $D'''=2(v_{a+1}+\cdots +v_{a+r+1})+E$ for some effective divisor $E$ on $\Gamma_0$.
From Lemma \ref{lemmaextra1} it follows $D'''$ is also linearly equivalent to $D$ as divisors on $\Gamma_0$.
This finishes the proof of Lemma \ref{lemma3}.

\subsection{The case $P \in \Gamma_0$}\label{subsection3.3}

Let $D$ be an effective divisor of degree $a+2r$ on $\Gamma_0$ and assume there exists $P \in \Gamma_0$ such that $\rk_{\Gamma} (D-P)=\rk_{\Gamma} (D)$.
Let $D'$ be an effective divisor on $\Gamma_0$ such that $D'+P$ is linearly equivalent to $D$ as a divisor on $\Gamma_0$.
From Lemma \ref{lemma2} we know $D'$ is linearly equivalent to $D-P$ as a divisor on $\Gamma$, hence $\rk_{\Gamma}(D')=r$.
Therefore we have to consider effective divisors $D'$ on $\Gamma_0$ of degree $a+2r-1$ such that $\rk_{\Gamma} (D') \geq r$.
The divisors $D$ we are considering are of the type $D'+P$ for such divisor $D'$ on $\Gamma_0$ and $P$ a point on $\Gamma_0$. 
So, let $D'$ be such a divisor on $\Gamma_0$.

For $a+1 \leq i \leq a+r-1$ we choose $P_i \in \gamma _i \setminus \{ v_i \}$ (hence in case $r=1$ this means we take no point, also $a+r-1 \leq g-2$ because $r \leq g-a-2$).
Also we choose points $P_g \in \gamma _g \setminus \{ v_g \}$ and $P_{g-1} \in \gamma _{g-1} \setminus \{ v_{g-1} \}$.
On $\Gamma$ there exists an effective divisor $D''$ linearly equivalent to $D'$ containing $P_{a+1} + \cdots + P_{a+r-1} + P_g$.
Repeating the arguments used in the proof of Lemma \ref{lemma3}, using Corollary \ref{corollary} we find $D'$ is linearly equivalent on $\Gamma_0$ to a divisor $D'''=2(v_{a+1} + \cdots + v_{a+r-1} + v_g)+E$ for some effective divisor $E$ on $\Gamma_0$ of degree $a-1$.
Using $P_{g-1}$ instead of $P_g$ we find the existence of an effective divisor $E'$ on $\Gamma_0$ of degree $a-1$ such that $2(v_{a+1} + \cdots + v_{a+r-1} + v_{g-1}) + E'$ is linearly equivalent to $D'$ on $\Gamma_0$.
It follows that $2v_g+E$ is linearly equivalent to $2v_{g-1}+E'$ on $\Gamma_0$, hence $\rk_{\Gamma_0} (2v_g+E-2v_{g-1})\geq 0$.

In case $E$ is linearly equivalent on $\Gamma_0$ to an effective divisor $F$ containing $v_{g-1}$ then $D=D'+P$ is linearly equivalent to $2(v_{a+1} + \cdots + v_{a+r-1} + v_g)+v_{g-1}+(F-v_{g-1}+P)$.
In that case $I(a+2r)(D)$ belongs to the subset $I(2r+1)(2(v_{a+1} + \cdots + v_{a+r-1} + v_g)+v_{g-1})+I(a-1)(\Gamma_0^{(a-1)})$ of $J(\Gamma_0)$.
This subset is the image of a finite number of bounded open subsets of affine subspaces of $\Omega (\Gamma_0)^*$ of dimension at most $a-1$.
A similar conclusion can be made in case $E$ would be linearly equivalent on $\Gamma_0$ to an effective divisor $F'$ containing $v_g$.
We can assume $I(a+2r)(D)$ does not belong to such subset of $J(\Gamma)$.

Now we consider the situation where $E$ is not linearly equivalent on $\Gamma_0$ to an effective divisor $F$ (resp. $F'$) containing $v_{g-1}$ (resp. $v_g$).
Using arguments as in Lemma \ref{lemma2} it is clear that each effective divisor on $\Gamma_0$ is linearly equivalent to an effective divisor whose support has at most one point in each $\gamma_i \setminus \{ w_i \}$ for $1 \leq i \leq a$ with $w_a=v_{g-1}$.
This implies $E$ is linearly equivalent on $\Gamma_0$ to a divisor $Q_1 + \cdots + Q_{a-1}$ with $Q_i \in \gamma_{j_i} \setminus \{ w_{j_i} \}$ for $1 \leq i \leq a-1$ with $1 \leq j_1 < j_2 < \cdots < j_{a-1} \leq a$ and $Q_1 \neq v_g$ in case $j_1=1$.
We are going to describe a certain finite subset $S$ of $\Gamma_0$ satisfying the following property.
If $Q_i \notin S$ for each $1 \leq i \leq a-1$ then  $\rk_{\Gamma_0} (2v_g+E-2v_{g-1})= -1$.
For $Q \in S$ and $Q_i = Q$ for some $1 \leq i \leq a-1$ one obtain $D=D'+P$ is linearly equivalent as a divisor on $\Gamma_0$ to $2(v_{a+1}+ \cdots + v_{a+r-1}+v_g)+Q+F'+P$ for some effective divisor $F'$ of degree $a-2$.
In that case $I(a+2r)(D)$ belongs to the subset $I(2r+1)(2(v_{a+1} + \cdots + v_{a+r-1} + v_g)+Q)+I(a-1)(\Gamma_0^{(a-1)})$ of $J(\Gamma_0)$.
Again this subset is the image of the union of finitely many bounded subsets of affine subspaces of $\Omega (\Gamma_0)^*$ of dimension at most $a-1$.
In case $I(a+2r)(D)$ does not belong to such subset of $J(\Gamma_0)$ then we obtain $\rk _{\Gamma} (D-P)<r$ for all $P\in \Gamma$.

Let $S$ be the finite set consisting of the following points: the points $Q'_1$ and $Q''_1$ on $\gamma_1$ such that $2v_g+Q'_1$ is linearly equivalent to $3w_1$ on $\gamma_1$ and $v_g+Q''_1$ is linearly equivalent to $2w_1$ on $\gamma_1$; 
for $2 \leq i \leq a-1$ the points $Q'_i$ and $Q''_i$ on $\gamma_i$ such that $2w_{i-1}+Q'_i$ is linearly equivalent to $3w_i$ and $w_{i-1}+Q''_i$ is linearly equivalent to $2w_i$ on $\gamma_i$
and the point $Q'_a$ on $\gamma_a$ such that $w_{a-1}+Q'_a$ is linearly equivalent to $2v_{g-1}$ on $\gamma_a$.
So we can assume $Q_i \notin S$ for $1\leq i\leq a-1$.
Define $0 \leq e\leq a-1$ such that $j_e=e$ and $j_{e+1}=e+2$ (no condition on $j_0$ in case $e=0$ and no condition on $j_a$ in case $e=a-1$).
We start with $2v_g+Q_1+\cdots +Q_{a-1}$.
In case $e>0$ then on $\gamma_1$ we obtain $2v_g+Q_1$ is linearly equivalent to $2w_1+R_1$ with $R_1 \notin \{ w_1,v_g \}$.
Hence because of Lemma \ref{lemma2} $2v_g+Q_1 + \cdots +Q_{a-1}$ is linearly equivalent to $R_1+2w_2+Q_2+ \cdots +Q_{a-1}$ on $\Gamma_0$.
Continuing in this way, in case $e=a-1$ we obtain $2v_g+E$ is linearly equivalent to a divisor $R_1+ \cdots \
+ R_{a_1}+2w_{a-1}$ on $\Gamma_0$ with $R_i \in \gamma_i \setminus \{w_{i-1},w_i \}$ for $1\leq i\leq a-1$ (here $w_0=v_g$).
Since $2w_{a-1}$ is not linearly equivalent to $2v_{g-1}$ on $\gamma_a$ we obtain $2v_g+E$ is linearly equivalent to $R_1 + \cdots +R_a + v_{g-1}$ with also $R_a \in \gamma_a \setminus \{w_{i-1},v_{g-1} \}$.
Using the burning algorithm on $\Gamma_0$ one easily sees that this divisor is $v_{g-1}$-reduced.
This implies $R_1+ \cdots + R_a - v_{g-1}$ is $v_{g-1}$-reduced too, hence $\rk_{\Gamma_0} (2v_g+E-2v_{g-1})=-1$ in this case.
In case $1\leq e<a-1$ then we obtain $2v_g+E$ is linearly equivalent to $R_1 + \cdots +R_e +2w_e + Q_{e+1} + \cdots + Q_{a-1}$ with $R_i \in \gamma_i \setminus \{ w_{i-1}, w_i \}$ in case $1 \leq i\leq e$.
Since $2v_g$ is linearly equivalent to $R_1+w_1$ with  $R_1 \in \gamma_1 \setminus \{ v_g,w_1 \}$ on $\gamma_1$ and $2w_{i-1}$ is linearly equivalent to $R_i+w_i$ with $R_i \in \gamma_i \setminus \{ w_{i-1},w_i \}$  on $\gamma_i$ for $2\leq i\leq a-1$ we obtain $2v_g+E$ is linearly equivalent to $R_1 + \cdots R_{e+1} + w_{e+1} + Q_{e+1} + \cdots + Q_{a-1}$ in case $0 \leq e < a-1$.
From this we obtain $2v_g+E$ is linearly equivalent on $\Gamma_0$ to a divisor $R_1 + \cdots + R_{a-1} + w_{a-1}+ Q_{a-1}$ with $R_i \in \gamma_i \setminus \{ w_{i-1},w_i \}$ for $e+2 \leq i \leq a-1$ and therefore it is linearly equivalent on $\Gamma_0$ to a divisor $R_1 + \cdots + R_a + v_{g-1}$ with $R_a \in \gamma_a \setminus \{ w_{a-1},v_{g-1} \}$.
As already mentioned before this is a $v_{g-1}$-reduced divisor, implying $\rk _{\Gamma_0} (2v_g+E-2v_{g-1})=-1$.

\end{document}